\documentclass{amsart}
\usepackage{graphicx}
\usepackage{amsmath}
\usepackage{amssymb}
\usepackage{latexsym}
\usepackage{tikz}
\usetikzlibrary{arrows,shapes,positioning}
\usetikzlibrary{decorations.markings}
\tikzstyle arrowstyle=[scale=1]
\tikzstyle directed=[postaction={decorate,decoration={markings,
		mark=at position .35 with {\arrow[arrowstyle]{angle 45}}}}]
\tikzstyle reverse directed=[postaction={decorate,decoration={markings,
		mark=at position .35 with {\arrowreversed[arrowstyle]{angle 45};}}}]

\newtheorem{lemma}{Lemma}
\newtheorem{theorem}{Theorem}
\newtheorem{corollary}{Corollary}

\newtheorem{conjecture}{Conjecture}

\newcommand{\Z}{\mathbb Z}
\newcommand{\F}{\mathbb F}
\newcommand{\FF}{\tilde{F}}

\begin{document}

\title{Decomposing almost complete graphs by random trees}

\author{A. Llad\'o}
 \address{Department of Mathematics \\
             Universitat Polit\`ecnica de Catalunya\\
              Barcelona, Spain}
\email{aina.llado@upc.edu}

\maketitle
\begin{abstract}
An old conjecture of Ringel states that every tree with $m$ edges decomposes the complete graph $K_{2m+1}$. The best known lower bound for the  order of a complete graph which admits a decomposition  by every given tree with $m$ edges is   $O(m^3)$. We show that asymptotically almost surely a random tree with $m$ edges and $p=2m+1$  a prime  decomposes $K_{2m+1}(r)$ for every $r\ge 2$, the graph obtained from the complete graph $K_{2m+1}$ by replacing each vertex by a coclique of order $r$.  Based on this result  we show, among other results,  that a random tree with $m+1$ edges a.a.s. decomposes the compete graph $K_{6m+5}$ minus one edge.

\end{abstract}

\section{Introduction}
\label{intro}
Given two graphs $H$ and $G$ we say that $H$ decomposes $G$  if $G$ is the edge--disjoint union of isomorphic copies of $H$. The following is a  well--known conjecture of Ringel.

\begin{conjecture}[Ringel \cite{ringel}]  Every tree with $m$ edges decomposes the complete  graph $K_{2m+1}$.\end{conjecture}

The conjecture has been verified by a number of particular classes of trees, see the extensive survey by Gallian \cite{gallian}. By using the polynomial method, the conjecture was verified by K\'ezdy \cite{kezdy06} for the more general class of  so--called {\it stunted} trees. As mentioned by the author, this class is still small among the set of all trees.

In this paper a random tree with $m $ edges is an unlabelled tree chosen uniformly at random among all the unlabelled trees with $m$ edges. 
Drmota and the author \cite{drmotallado13} used structural results on random trees to show that asymptotically almost surely (a.a.s.) a random tree with $m$ edges decomposes the complete bipartite graph $K_{2m,2m}$, thus providing an aproximate result to another decomposition conjecture by Graham and Haggkvist which asserts that  $K_{m,m}$ can be decomposed by any  given tree with $m$ edges.

Results in this vein have been also recently obtained by B\"ottcher, Hladk\'y, Piguet and Taraz \cite{BHPT2014}, where the authors show that, for any $\epsilon>0$ and  any $\Delta$, every family of trees with   maximum degree at most $\Delta$ and at most ${n \choose 2}$ edges in total  packs  into the  complete graph $K_{(1+\epsilon )n}$ for every sufficiently large $n$.

Coming back to decompositions, or perfect packings, let $g(m)$ be the smallest integer $n$ such that any tree with $m$ edges decomposes the complete graph $K_{n}$. It was shown by  Yuster \cite{yuster} that $g(m)=O(m^{10})$ and the upper bound was reduced by Kezdy and Snevily \cite{kezdy02} to $g(m)=O(m^3)$. Since $K_{2m,2m}$ decomposes the complete graph $K_{8m^2+1}$ (see Snevily \cite{snevily}), the above mentioned result on the decomposition of $K_{2m,2m}$ shows that $g(m)=O(m^2)$ asymptotically almost surely.

In this paper we  prove that one can decompose almost complete graphs by random trees, getting much closer to the original conjecture of Ringel.

For  positive integers $n, r$ we denote by $K_{n}(r)$ the blow--up graph obtained from the complete graph $K_n$ by replacing each vertex by a coclique with order $r$ and joining every pair of vertices which do not belong to the same coclique. Our main result is the following one.

\begin{theorem}\label{thm:main1}
For any $m$ such that $p=2m+1$ is a prime, and every $r\ge 2$, asymptotically almost surely a random tree with $m$ edges decomposes $K_{2m+1}(r)$.
\qed
\end{theorem}

As an application of Theorem  1  we obtain  the following corollaries, which are   approximate results for random trees  of Ringel's conjecture.

The following statement is a direct consequence of Theorem \ref{thm:main1} with $r=2$.

\begin{corollary}
\label{cor:4m+2}
For every $m$ such that $p=2m+1$ is a prime, asymptotically almost surely,  a random tree with $m$ edges decomposes $K_{4m+2}\setminus M$, where $M$ is any perfect matching.
\qed
\end{corollary}

Next Corollaries  also follow from Theorem \ref{thm:main1} with some additional work.

\begin{corollary}\label{cor:6m+5}
For every $m$ such that $p=2m+1$ is a prime,  a random tree with  $m+1$ edges a.a.s. decomposes  $K_{6m+5}\setminus e$, where $e$  is an edge of the complete graph.
\label{6m+5}
\qed
\end{corollary}

By using similar techniques as the ones involved in the proof of Corollary \ref{cor:6m+5}   the result  can be extended to any odd $r\ge 3$.

\begin{corollary}\label{cor:r>3}
For each odd number $r\ge 3$ and every $m$ such that $p=2m+1$ is a prime a random tree with  $m+1$ edges a.a.s. decompose
$$K_{n}\setminus K_{t}.$$
where $t=(r+1)/2$ and $n=r(2m+1)+t$.
\qed
\label{r>3}
\end{corollary}
This extension of Corollary \ref{cor:6m+5} can be seen as an approximation to a more general conjecture by Ringel which states that every tree with $m$ edges decomposes the complete graph $K_{rm+1}$ whenever   $r$ and $m$ are not both odd.

The paper is organised as follows. In Section \ref{sec:rainbow} we introduce the notion of rainbow embeddings in connection to graph decompositions and give some results which provide a rainbow embedding of a given tree  in an appropriate Cayley graph. The embedding techniques use the polynomial method of Alon and bring the condition that $p=2m+1$ is a prime in the statement of Theorem \ref{thm:main1}. These techniques are not enough to ensure that the rainbow embedded copy is isomorphic to the given tree. In order to complete the result we consider the blow up of the complete graph and perform some local modifications of the rainbow embedding in Section \ref{sec:dec}. The proofs of Theorem \ref{thm:main1} and of the Corollaries \ref{cor:4m+2}, \ref{6m+5} and \ref{r>3} are given in Section \ref{sec:proofs}.


\section{Rainbow embeddings}
\label{sec:rainbow}

The general approach to show that a tree $T$ decomposes a complete graph  consists in showing that $T$ cyclically decomposes the corresponding graph. We next recall the basic principle behind this approach in slightly different terminology.

A rainbow embedding of a graph $H$ into an oriented  arc--colored graph $X$ is  an injective homomorphism $f$ of some orientation $\overrightarrow H$ of $H$ in $X$ such that no two arcs of $f(\overrightarrow{H})$ have the same color. According to its  common use, even if a rainbow embedding  is meant to be defined as a map $f:V(H)\to V(X)$, we still call $f$ the induced map   $f:E(\overrightarrow{H})\to E(X)$ on arcs   defined as $f(x,y)=(f(x),f(y))$, and we think of $f$ as a map $f:\overrightarrow{H}\to X$.

Let $X=Cay(G,S)$ be a Cayley digraph on an abelian group $G$ with respect to an antisymmetric subset $S\subset G$  ($S\cap -S=\emptyset$.) We consider $X$ as an arc--colored oriented graph, by giving to each arc $(x,x+s)$, $x\in G, s\in S$, the color $s$.

\begin{lemma}
If a graph $H$   admits a rainbow embedding  in $X=Cay(G,S)$, where $S$ is an antisymmetric subset of a group $G$ of order $n$ then the underlying graph of $X$ contains $n$ edge--disjoint copies of $H$. In particular, if $H$ has $|S|$ edges then $H$ decomposes the underlying graph of $X$.
\end{lemma}

\begin{proof}
Let $f:H\to X$ be a rainbow embedding. For each $a\in G$ the translation  $x\to x+a$, $x\in G$, is an automorphism of $X$ which preserves the colors and has no fixed points. Therefore, each translation sends $f(\overrightarrow{H})$ to an isomorphic copy which is  edge--disjoint from it. Thus the sets of translations for all $a\in G$ give rise to $n$ edge--disjoint copies of $\overrightarrow{H}$ in $X$. By ignoring orientations and colors, we thus have $n$ edge disjoint copies of $H$ in the underlying graph of $X$. In particular, if $H$ has $|S|$ edges then $H$ decomposes the underlying graph of $X$.
\qed
\end{proof}

The proof of the main Theorem uses the above Lemma for a rainbow subgraph of an appropriate Cayley graph $X$. Instead of finding a rainbow embedding of a tree $T$ we will find a rainbow edge--injective homomorphism of $T$ in $X$ in two steps, first embedding $T_0$, the tree with some  leaf s removed, and then embedding the remaining forest $F$ of stars to complete $T$.

For the first step we use the the so--called Combinatorial Nullstellensatz of Alon \cite{alon1} that we next recall.

\begin{theorem}[Combinatorial Nullstellensatz]
\label{thm:alon2}
Let  $P\in F[x_1,\ldots ,x_k]$ be a polynomial of degree $d$ in $k$ variables with coefficients in a field $F$.

If the coefficient of the monomial $x_1^{d_1}\cdots x_k^{d_k}$, where $\sum_id_i=d$, is nonzero,  then $P$ takes a nonzero value in every grid $A_1\times \cdots\times A_k$ with $|A_i|>d_i$, for $1\le i\le k$.\qed
\end{theorem}

In the following Lemma we   use Theorem \ref{thm:alon2} in a way inspired by K\'ezdy  \cite{kezdy06}. A {\it peeling ordering} of  a tree $T$ is an ordering $x_0,\ldots ,x_m$ of $V(T)$ such that for every $0\le t\le m$ the induced  subgraph $T[x_0,\ldots ,x_t]$   is a subtree of $T$. We assume that $T$ is an oriented tree with all its edges oriented from the root $x_0$.

Next Lemma shows that any tree with $k$ edges  can  be rainbowly embedded in some $Cay(\Z_p,S)$ with $|S|=k$ provided that $k$ is not too large with respect to $p$.

 \begin{lemma}
 \label{lem:base}
 Let $p>10$ be  a prime and  $T$ a tree with $k<3(p-1)/10$ edges. There is an antisymmetric set $S\subset \Z_p^*$ with $|S|=k$ such that  $T$ admits a rainbow embedding in $Cay(\Z_p,S)$.
\end{lemma}

\begin{proof}
Let $x_0,x_1,\ldots ,x_k$ be a peeling  ordering of  $T$. Label the edges of $T$ by variables  $y_1, \ldots ,y_k$  such that the edge labelled $y_i$ joins $x_i$ with $T[x_0,x_1,\ldots ,x_{i-1}]$, $0<i\le k$. For each $i$ we denote by $T(0,i)$ the set of subscripts $j$ such that the edges $y_j$  lie in the unique path from $x_0$ to $x_i$ in $T$. Consider the polynomial $P\in \F_p[y_1,\ldots ,y_k]$ defined as
$$
P(y_1,\ldots ,y_k)=\prod_{1\le i<j\le k} (y_j^2-y_i^2)\prod_{1\le i<j\le k} (\sum_{r\in T(0,i)} y_r-\sum_{s\in T(0,j)} y_s),
$$
which has degree $2{k\choose 2}+{k\choose 2}=3k(k-1)/2$.

Suppose that $P(a_1,a_2,\cdots,a_k)\neq 0$ for some point $(a_1,\ldots ,a_k)\in (\F_p^*)^k$.   Then, since the first factor $Q= \prod_{i<j} (y_i^2-y_j^2)$ of $P$ is nonzero at $(a_1,\ldots ,a_k)$, we have $a_i\neq \pm a_j$ for each pair $i\neq j$. Hence the set $S=\{ a_1,\ldots ,a_k\}$ consists of pairwise distinct elements and it is antisymmetric.

Moreover, since the second factor $R=\prod_{i<j} (\sum_{y_r\in T(0,i)} y_r-\sum_{y_r\in T(0,j)} y_r)$ is nonzero at $(a_1,\ldots ,a_k)$, then the map $f:V(T)\to Cay(\Z_p,S)$  defined as $f(x_0)=0$ and 
$$f(x_i)=\sum_{r\in T(0,i)}y_r,$$
for  $1\le i\le k$, is injective and provides a rainbow embedding of $T$ in $Cay(\Z_p, S)$.

\

Let us show that $P$ is nonzero at some point of $(\Z_p^*)^k$. To this end we consider the monomial of maximum degree
$$
y_k^{3(k-1)}y_{k-1}^{3(k-2)}\cdots y_1^0,
$$
in $P$. This monomial can be obtained in the expansion of $P$ by collecting $y_k$ in all the factors of $Q$ where it appears, giving $y_k^{2(k-1)}$, and also in all terms of $R$ where it appears, which, since $y_k$ is a  leaf  of $T$, gives $y_{k}^{k-1}$. This is the unique way to obtain $y_k^{3(k-1)}$ in a monomial of $P$. Thus the coefficient of $y_k^{3(k-1)}$ in $P$ is
$$
[y^{3(k-1)}]P=\pm P_{k-1},
$$
where
$$
P_{k-1}(y_1,\ldots ,y_{k-1})=\prod_{1\le i<j\le k-1} (y_i^2-y_j^2)\prod_{1\le i<j\le k-1} (\sum_{r\in T(0,i)} y_r-\sum_{s\in T(0,j)} y_s).
$$
By iterating the same argument we conclude that the coefficient in $P$  of
$$y_k^{3(k-1)}y_{k-1}^{3(k-2)}\cdots y_1^0$$
is $\pm 1$ and, in particular, different from zero. Since $3(k-1)<9p/10<p-1$ for $p>10$, we conclude from Theorem \ref{thm:alon2} that $P$ takes a nonzero value in $(\Z_p^*)^k$. This concludes the proof.
\qed
\end{proof}

In the second step we try  to embed rainbowly a forest of stars. We still use Theorem \ref{thm:alon2}, or rather the following  consequence derived from it by Alon \cite{alon2}.

\begin{theorem}[Alon  \cite{alon2}]
 \label{thm:alon1}
Let $p$ be a prime and $k< p$. For every sequence $a_1,\ldots ,a_k$ (possibly with repeated elements) and every set $\{b_1,\ldots ,b_k\}$ of elements  of  $\Z_p$ there is a permutation $\sigma\in Sym (k)$ such that the sums $a_1+b_{\sigma (1)},\ldots ,a_k+b_{\sigma (k)}$ are pairwise distinct.
\qed
\end{theorem}

The rainbow map defined with the help of Theorem \ref{thm:alon1} may fail to be a rainbow embedding of the forest because some endvertices may be sent to some center of another star.
One consequence of the above result is that,  for every antisymmetric set $S\subset \Z_p$ with $h$ elements, every forest of stars with $h$ edges  admits a rainbow `quasi' embedding in $Cay(\Z_p, S)$. Moreover, the centers of the stars in the forest can be placed at prescribed vertices.
The following is the precise statement.

\begin{lemma}
\label{lem: leaf s}
Let $p$ be a prime. Let $F$ be a forest of $k$ stars centered at $x_1,\ldots ,x_k$ and $h\le (p-1)/2$ edges. Let $S\subset \Z_p^*$ be an antiymmetric set with $|S|=h$.

Every injection $f:\{x_1,\ldots ,x_k\}\to \Z_p$ can be extended to a rainbow edge-injective homomorphism,
$f_1:\;F\longrightarrow Cay(\Z_p,S)$   such  that  $f_1(F)$ is an oriented graph with maximum indegree one.
\end{lemma}

\begin{proof}
Consider the sequence $(f(x_1)^{h_1}, \ldots , f(x_k)^{h_k})$, where the multiplicity $h_i$ of $f(x_i)$ is the number of  leaf s of the star centered at $x_i$, $\sum_ih_i=h$.

By Theorem \ref{thm:alon1} there is a numbering $s_1,\ldots ,s_h$ of the elements of $S$ such that  for any $\; 1\le i \le k$ and any
$\sum_{r=1}^{i-1} h_r<j\le\sum_{r=1}^{i} h_r,$ the sums
$$
f(x_i)+s_j,
$$
are pairwise distinct.

Label the noncenter vertices of $F$ by $y_1,\ldots ,y_k$,  such that $y_j$ adjacent to $x_i$ whenever
$\sum_{r=1}^{i-1} h_r<j\le\sum_{r=1}^{i} h_r,$ and orient the edges of $F$ from the centers of the stars to their endvertices.

For each $i$ and each $\sum_{r=1}^{i-1} h_r<j\le\sum_{r=1}^{i} h_r,$,  we obtain the desired rainbow embedding by defining,
 $$f_1(x_i)=f(x_i), \hspace{4mm}f_1(y_{j})=f(x_i)+s_j.$$
 Since all sums are distinct,  no two endvertices  of $F$ are sent to the same vertex by $f_1$ and each of them has indegree one in $f_1(F)$; by the same reason, every $f_1(x_i)$ can coincide with at most one $f_1(y_{j})$ for some $y_{j}$ not in the same star as $x_i$. Thus the image $f_1(F)$ has indegree at most one.\qed
\end{proof}


\section{The decomposition}\label{sec:dec}

In this Section we prove Theorem \ref{thm:main1}.  The strategy of the proof is as follows. We decompose the given tree $T$ into a tree $T_0$ and a forest $F$ of stars centred at some vertices of $T_0$,
$$
T=T_0\oplus F.
$$
We embed $T$ rainbowly in $Cay (\Z_p,S)$ where $S\subset \Z_p^*$ is an antisymmetric set of cardinality $|S|=(p-1)/2$. We embed $T_0$ and $F$ by using Lemma \ref{lem:base} and Lemma \ref{lem: leaf s}. Doing so, the image of $T$ by the rainbow embedding may  be nonisomorphic to $T$. The last step in the proof consists in extending the rainbow embedding to $Cay(\Z_p\times \Z_r, S\times \Z_r)$ and rearranging some arcs to obtain a decomposition of this oriented graph into copies of $T$.

For a  graph $G$ and a positive integer $r$ we denote by $G(r)$ the graph obtained form $G$ by replacing each vertex with a coclique of order $r$ and every vertex is adjacent to all vertices except the ones in their coclique (below $K_5(3)$  illustrates the definition).
The same notation is used when $G$ is an oriented graph.

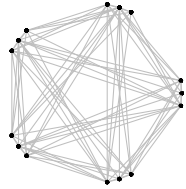
\begin{center}
\begin{tikzpicture}[scale=0.3]
\foreach \i in {1,...,5}
{
\path (72*\i:4cm) coordinate (P\i);
\path (72*\i+8:4cm) coordinate (Q\i);
\path (72*\i-8:4cm) coordinate (R\i);
}
\foreach \i in {1,...,5}
{
\foreach \j in {\i,...,5}
{
\draw[lightgray] (P\i)--(P\j) (P\i)--(Q\j) (P\i)--(R\j);
\draw[lightgray] (Q\i)--(Q\j) (Q\i)--(R\j) (R\i)--(R\j);
\foreach \i in {1,...,5}
{
\draw[fill]  (P\i) circle (2pt);
\draw[fill]  (Q\i) circle (2pt);
\draw[fill]  (R\i) circle (2pt);
}
}
}
\end{tikzpicture}
\end{center}

We will need the following technical Lemma.

\begin{lemma}\label{lem:mamb}
Let $r\ge 2$ be an integer and let
$$
M=(M_a, M_b)=\left(
      \begin{array}{cccc|cccc}
        1 & 2 &\cdots & r &\sigma_1 & \sigma_2 &\cdots & \sigma_r \\
        r+1& r+2& \cdots & 2r& \sigma_{r+1}& \sigma_{r+2}& \cdots & \sigma_{2r} \\
        \vdots & &  & \vdots & \vdots & &  & \vdots\\
        r(r-1)+1&r(r-1)+2&\cdots &r^2&  \sigma_{r(r-1)+1}&\sigma_{r(r-1)+2}&\cdots &\sigma_{r^2}
      \end{array}
    \right).
$$
be a  matrix where $(\sigma_1,\ldots ,\sigma_{r^2})$ is a permutation of $\{1,\ldots ,r^2\}$.

Then, there are permutations of the  elements in each column of $M$ in such a way the resulting matrix $M'$  has no row with repeated entries.
\end{lemma}

\begin{proof}
We proceed row by row.
By the definition of $M$, each column has $r$ distinct entries. Let $M_{a,i}$  be the set of entries in the $i$-th column of $M_a$ and   $M_{b,j}$  be the set of entries in the $j$-th column of $M_b$.

We use Hall's theorem to find a transversal of the family  of 
$${\mathcal M}=\{M_{a1},M_{a2},\cdots M_{ar},M_{b1},M_{b2},\cdots M_{br}\}.$$
For each pair of subsets $I,J\subset \{1,2,\ldots ,r\}$  we have,
\begin{equation}\label{eq:trans}
|I|+|J|\le2\max\{ |I|,|J|\}\le r\max \{|I|,|J|\}\le|(\cup_{i\in I}M_{a,i})\cup (\cup_{j\in J} M_{b,j})|
\end{equation}
which shows that Hall's condition holds and therefore ${\mathcal M}$ has a transversal. We place this transversal in the first row of the new matrix $M'$.

By deleting each element of the transversal from its set of ${\mathcal M}$  we get a family of $(r-1)$--sets for which the  inequalities in \eqref{eq:trans} hold with $r$ replaced by $(r-1)$ as long as $r-1\ge 2$. Hence there is a transversal of this new family of sets which we place in  the second row of $M'$. We can proceed with the same argument up to the $(r-1)$ row. Now if each of the first $r-1$ rows of $M'$ have their entries pairwise distinct,  the remaining elements are also pairwise distinct and can be placed in the last row of $M'$.
\qed
\end{proof}

We next proceed to the main Lemma in this Section.

 \begin{lemma}
 \label{lem:3m}
 Let $p>10$ be a prime and $r\ge 2$ an integer. Let $T$ be a tree with $m=(p-1)/2$ edges and at least $2m/5$  leaf s. Then $T$ decomposes $K_{2m+1}(r)$.
 \end{lemma}

\begin{proof}
Remove $\lceil 2m/5\rceil$  leaves from $T$  and denote by $T_0$ the resulting tree. Let $F$ be the forest of stars with centers in vertices of $T_0$ so that
$$
T=T_0\oplus F.
$$
We split the proof of the Lemma into three steps.

\

{\bf Step 1.}
Define a rainbow embedding of $T$ into $X=Cay (\Z_p,S)$ where $S\subset \Z_p^*$  is an antisymmetric set with $|S|=(p-1)/2$. By ignoring the orientations of $X$ we obtain the complete graph $K_p$.

 Let $t\le 3m/5<3(p-1)/10<(p-1)/3$ be the number of edges of  $T_0$. By Lemma \ref{lem:base},    there is an antisymmetric subset $S_0\subset \Z_p^*$ with $|S_0|=t$ and a   rainbow embedding
$$
f_0:T_0\to Cay(\Z_p, S_0).
$$
Let $x_0,\ldots ,x_t$ be a peeling ordering of $T_0$.  Since $t>\lceil 2m/5\rceil$, we may assume that $x_0$ is not incident with a  leaf  in $F$. By exchanging elements of $S$ by their opposite ones
 if necessary, we may assume that $f_0(T_0)$ has all its edges oriented from $x_0$ to the  leaves of $T_0$. By abuse of notation we still denote by $x_0,\ldots ,x_t$ the images of the vertices of $T_0$ by $f_0$.
We may assume that $f_0(x_0)=0$.

 Let $S$ be an antisymmetric subset of $\Z_p^*$ with $|S|=(p-1)/2$ which contains $S_0$, so that $|S-S_0|=|E(F)|$. Let $x_{i_1}=v_1,\cdots,, x_{i_k}=v_k$ be the centers of the stars of $F$. By Lemma \ref{lem: leaf s} there is an edge--injective rainbow homomorphism  of the  forest $F$  into $Cay (\Z_p, S\setminus S_0)$,
$$
f_1:F\to Cay(\Z_p,S\setminus S_0),
$$
 such that  $f_1(v_i)=f_0(v_i)$, where $v_1,\cdots, v_k$ are  the centers   of the stars of $F$.
  Moreover $\FF=f_1(F)$ is an oriented graph with maximum indegree one.

The  map  $f:V(T)\to Cay (\Z_p,S)$ defined by $f_0$ on $V(T_0)$ and by $f_1$ on $V(F)$  is well defined, since  $f_1(v_i)=f_0(v_i)$, and
$$f(T)=f_0(T_0)\oplus f_1(F)=H$$
is a rainbow subgraph of $X=Cay(\Z_p,S)$.

We note that $f$ may fail to be a rainbow embedding of $T$ in $X=Cay (\Z_p,S)$ to the effect  that  some endvertices of $T$ may have been  sent by $f_1$ to some vertices of $f_0(V(T_0))$. Thus $H$ may be not isomorphic to $T$ and contain some cycles (see Figure \ref{cycles} for an illustration.)
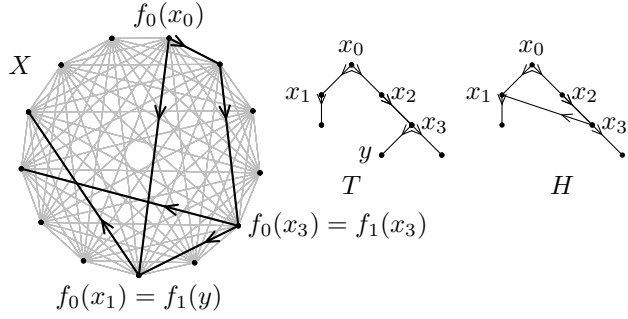
\begin{figure}[h]
\begin{center}
\begin{tikzpicture}[scale=0.4]
\foreach \i in {1,...,13}
{
\path (27.5*\i:4cm) coordinate (P\i);
\path (27.5*\i+6:4cm) coordinate (P\i);
\path (27.5*\i-6:4cm) coordinate (P\i);
\draw[fill]  (P\i) circle (2pt);
}
\foreach \i in {1,...,12}
{
\foreach \j in {\i,...,13}
{
\draw[lightgray] (P\i)--(P\j) (P\i)--(P\j) (P\i)--(P\j);
\draw[lightgray] (P\i)--(P\j) (P\i)--(P\j) (P\i)--(P\j);
}
}
\foreach \i in {1,...,13}
{
\draw[fill]  (P\i) circle (2pt);
}
\path (7,3) coordinate (S0);
\path (6,2) coordinate (S1);
\path (8,2) coordinate (S2);
\path (6,1) coordinate (S3);
\path (9,1) coordinate (S4);
\path (9,1) coordinate (S5);
\path (8,0) coordinate (S6);
\path (10,0) coordinate (S7);
\foreach \i in {0,...,7}
{
\draw[fill] (S\i) circle (2pt);
}
\draw[directed] (S0)--(S1);
\draw[directed] (S0)--(S2);
\draw[directed]  (S1)--(S3);
\draw[directed]  (S2)--(S4);
\draw[directed]   (S2)--(S5);
\draw[directed]  (S5)--(S6);
\draw[directed]  (S5)--(S7);
\node at (7,-1) {$T$};
\node[above] at (S0) {$x_0$};
\node[left] at (S1) {$x_1$};
\node[right] at (S2) {$x_2$};
\node[right] at (S4) {$x_3$};
\node[left] at (S6) {$y$};
\node at (-4,3) {$X$};
\draw[ thick,  directed] (P3)--(P2);
\draw[thick,  directed] (P2)--(P12);
\draw[thick, directed] (P12)--(P10);
\draw[thick, directed] (P10)--(P6);
\draw[thick, directed] (P12)--(P7);
\draw[thick,  directed] (P3)--(P10);
\node[above] at (P3) {$f_0(x_0)$};
\node[right] at (P12) {$f_0(x_3)=f_1(x_3)$};
\node[below] at (P10) {$f_0(x_1)=f_1(y)$};
\foreach \i in {0,...,7}
{
\draw[fill] (S\i) circle (2pt);
}
\foreach \i in {1,...,13}
{
\draw[fill]  (P\i) circle (2pt);
}
\path (13,3) coordinate (T0);
\path (12,2) coordinate (T1);
\path (14,2) coordinate (T2);
\path (12,1) coordinate (T3);
\path (15,1) coordinate (T4);
\path (15,1) coordinate (T5);
\path (14,0) coordinate (T6);
\path (16,0) coordinate (T7);
\foreach \i in {0,...,5}
{
\draw[fill] (T\i) circle (2pt);
}
\draw[fill] (T7) circle (2pt);
\draw[directed] (T0)--(T1);
\draw[directed] (T0)--(T2);
\draw[directed]  (T1)--(T3);
\draw[directed]  (T2)--(T4);
\draw[directed]   (T2)--(T5);
\draw[directed]  (T5)--(T1);
\draw[directed]  (T5)--(T7);
\node at (14,-1) {$H$};
\node[above] at (T0) {$x_0$};
\node[left] at (T1) {$x_1$};
\node[right] at (T2) {$x_2$};
\node[right] at (T4) {$x_3$};
\end{tikzpicture}
\end{center}
\caption{A rainbow map of $T$ with conflicting arcs at $f_0(x_1)=f_1(y)$.}
\label{cycles}
\end{figure}

 We observe however that,  if $f_1(y)=f_0(x)$ for some endvertex $y\in F$ and some $x\in V(T_0)$, then  $y$ is not adjacent to $x$ in $T$ because $f_1$ is an edge--injective homomorphism. In other words, $f(T)$ has maximum indegree at most two.

 \

{\bf Step 2.} Extending the rainbow map to $X(r)=Cay (\Z_p\times \Z_r, S\times \Z_r)$.

Let  $Y=X(r)$. By ignoring the orientations and colors
of the arcs in $Y$ we obtain $K_p(r)$.

For each pair $i,j\in \Z_r$ we define a subgraph $H_{ij}$ of $Y$ as the
image by the injective homomorphism $f_{ij}:H\to Y$ such that
$f_{ij}(0)=(0,i)$ and every arc $(x,x+s)\in E(H)$ is sent to the arc
$(f_{ij}(x), f_{ij}(x)+(s,j))$ of $E(Y)$.

Since $H$ is a connected subgraph of $X$, the map $f_{ij}$ is well defined.
Moreover $H_{ij}=f_{ij}(H)$ is a rainbow subgraph of $Y$.
Below there is an illustration for $p=13$ of the subgraphs $H_{0,j}$ corresponding to the example of Figure \ref{cycles}.
\begin{figure}[h]
\begin{center}
{\tiny
\begin{tikzpicture}[scale=0.4]
\foreach \i in {0,...,5}
{
}
\path (13,6) coordinate (T0);
\path (12,4) coordinate (T1);
\path (14.5,4) coordinate (T2);
\path (12,2) coordinate (T3);
\path (15,2) coordinate (T4);
\path (15,2) coordinate (T5);
\path (15,0) coordinate (T7);
\draw[fill] (T7) circle (3pt);
\draw[fill] (T0) circle (3pt);
\draw[fill] (T1) circle (3pt);
\draw[fill] (T2) circle (3pt);
\draw[fill] (T3) circle (3pt);
\draw[fill] (T4) circle (3pt);
\draw[fill] (T5) circle (3pt);
\draw[directed] (T0)--(T1);
\draw[directed] (T0)--(T2);
\draw[directed]  (T1)--(T3);
\draw[directed]  (T2)--(T4);
\draw[directed]   (T2)--(T5);
\draw[directed]  (T5)--(T1);
\draw[directed]  (T5)--(T7);
\node at (13,1) {$H=H_{00}$};
\node at (15.8,1) {$(5,0)$};
\node at (15.7,3) {$(3,0)$};
\node at (14.7,5) {$(1,0)$};
\node at (11.7,5) {$(6,0)$};
\node at (11,3) {$(4,0)$};
\node at (13.3,2.4) {$(2,0)$};
\node[above] at (T0) {$(0,0)$};
\node[left] at (T1) {$(6,0)$};
\node[right] at (T2) {$(1,0)$};
\node[right] at (T4) {$(4,0)$};
\node[right] at (T7) {$(9,0)$};
\node[left] at (T3) {$(10,0)$};
\end{tikzpicture}
\begin{tikzpicture}[scale=0.4]
\foreach \i in {0,...,5}
{
}
\path (13,6) coordinate (T0);
\path (12,4) coordinate (T1);
\path (13,4) coordinate (T6);
\path (14.5,4) coordinate (T2);
\path (12,2) coordinate (T3);
\path (15,2) coordinate (T4);
\path (15,2) coordinate (T5);
\path (13,3.2) coordinate (T6);
\path (15,0) coordinate (T7);
\draw[fill] (T6) circle (3pt);
\draw[fill] (T7) circle (3pt);
\draw[fill] (T0) circle (3pt);
\draw[fill] (T1) circle (3pt);
\draw[fill] (T2) circle (3pt);
\draw[fill] (T3) circle (3pt);
\draw[fill] (T4) circle (3pt);
\draw[fill] (T5) circle (3pt);
\draw[directed] (T0)--(T1);
\draw[directed] (T0)--(T2);
\draw[directed]  (T1)--(T3);
\draw[directed]  (T2)--(T4);
\draw[directed]   (T2)--(T5);
\draw[directed]  (T4)--(T6);
\draw[directed]  (T5)--(T7);
\node at (13,1) {$H_{01}$};
\node at (15.8,1) {$(5,1)$};
\node at (15.7,3) {$(3,1)$};
\node at (14.7,5) {$(1,1)$};
\node at (11.5,5) {$(6,1)$};
\node at (11,3) {$(4,1)$};
\node at (13.5,2.2) {$(2,1)$};
\node[above] at (T6) {$(6,3)$};
\node[above] at (T0) {$(0,0)$};
\node[left] at (T1) {$(6,1)$};
\node[right] at (T2) {$(1,1)$};
\node[right] at (T4) {$(4,2)$};
\node[right] at (T7) {$(9,3)$};
\node[left] at (T3) {$(10,2)$};
\end{tikzpicture}
\begin{tikzpicture}[scale=0.4]
\foreach \i in {0,...,5}
{
}
\path (13,6) coordinate (T0);
\path (12,4) coordinate (T1);
\path (14.5,4) coordinate (T2);
\path (12,2) coordinate (T3);
\path (15,2) coordinate (T4);
\path (15,2) coordinate (T5);
\path (15,0) coordinate (T7);
\draw[fill] (T7) circle (3pt);
\draw[fill] (T0) circle (3pt);
\draw[fill] (T1) circle (3pt);
\draw[fill] (T2) circle (3pt);
\draw[fill] (T3) circle (3pt);
\draw[fill] (T4) circle (3pt);
\draw[fill] (T5) circle (3pt);
\draw[directed] (T0)--(T1);
\draw[directed] (T0)--(T2);
\draw[directed]  (T1)--(T3);
\draw[directed]  (T2)--(T4);
\draw[directed]   (T2)--(T5);
\draw[directed]  (T5)--(T1);
\draw[directed]  (T5)--(T7);
\node at (13,1) {$H_{02}$};
\node at (15.8,1) {$(5,2)$};
\node at (15.7,3) {$(3,2)$};
\node at (14.7,5) {$(1,2)$};
\node at (11.6,5) {$(6,2)$};
\node at (11,3) {$(4,2)$};
\node at (13.4,2.2) {$(2,2)$};
\node[above] at (T0) {$(0,0)$};
\node[left] at (T1) {$(6,2)$};
\node[right] at (T2) {$(1,2)$};
\node[right] at (T4) {$(4,0)$};
\node[right] at (T7) {$(9,2)$};
\node[left] at (T3) {$(10,0)$};
\end{tikzpicture}
\begin{tikzpicture}[scale=0.4]
\foreach \i in {0,...,5}
{
}
\path (13,6) coordinate (T0);
\path (12,4) coordinate (T1);
\path (14.5,4) coordinate (T2);
\path (12,2) coordinate (T3);
\path (15,2) coordinate (T4);
\path (15,2) coordinate (T5);
\path (15,0) coordinate (T7);
\path (13,3.2) coordinate (T6);
\draw[fill] (T6) circle (3pt);
\draw[fill] (T7) circle (3pt);
\draw[fill] (T0) circle (3pt);
\draw[fill] (T1) circle (3pt);
\draw[fill] (T2) circle (3pt);
\draw[fill] (T3) circle (3pt);
\draw[fill] (T4) circle (3pt);
\draw[fill] (T5) circle (3pt);
\draw[directed] (T0)--(T1);
\draw[directed] (T0)--(T2);
\draw[directed]  (T1)--(T3);
\draw[directed]  (T2)--(T4);
\draw[directed]   (T2)--(T5);
\draw[directed]  (T4)--(T6);
\draw[directed]  (T5)--(T7);
\node at (13,1) {$H_{03}$};
\node at (15.8,1) {$(5,3)$};
\node at (15.6,3) {$(3,3)$};
\node at (14.7,5) {$(1,3)$};
\node at (11.5,5) {$(6,3)$};
\node at (11.2,3) {$(4,3)$};
\node at (13.6,2.1) {$(2,3)$};
\node[above] at (T6) {$(6,1)$};
\node[above] at (T0) {$(0,0)$};
\node[left] at (T1) {$(6,3)$};
\node[right] at (T2) {$(1,3)$};
\node[right] at (T4) {$(4,2)$};
\node[right] at (T7) {$(9,1)$};
\node[left] at (T3) {$(10,2)$};
\end{tikzpicture}}
\end{center}
\caption{The rainbow subgraphs $H_{0,j}=f_{0j}(H)$ of $X(4)=Cay (\Z_{13}\times \Z_4, S\times \Z_4)$.}
\label{conflict}
\end{figure}
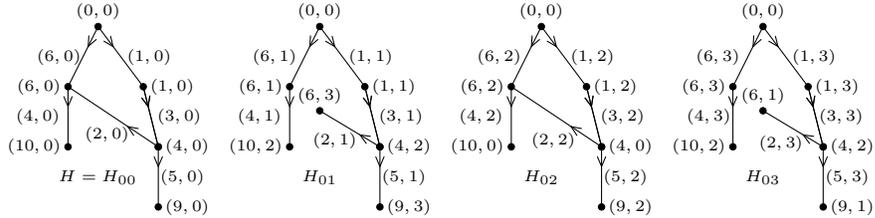

Every  $H_{ij}$ can be decomposed into
$$
H_{ij}=T_{ij}\oplus F_{ij},
$$
where, since $T_0$ is acyclic,  $T_{ij}$ is isomorphic to $T_0$. As in the Step 1, $H_{ij}$ may  be non isomorphic to the original tree $T$, but only due to the  fact that some end vertex  of $F_{ij}$ may have   been identified with  some vertex of $T_{ij}$. However, the in--degree of every vertex in $H_{ij}$ is again at most two as this was the case in $H$. If there is a vertex with indegree two in $H_{ij}$ we call its incoming arcs to be {\it conflicting}.

We note that the $H_{ij}$'s are edge-disjoint (they hold pairwise distinct labels for $j$ fixed and these labels emerge from distinct vertices for each $i$).
Let
$$
H_i=\oplus_{0\le j<r} H_{ij}\;\; \text{and} \;\; H(r)=\oplus_{0\le i< r} H_{i}=\oplus_{0\le i,j< r} H_{ij}.
$$
By the definition of $f_{ij}$, we observe that each $H_i$ is a rainbow subgraph of $Y$ with $r(p-1)/2$ edges, so that all colors of the generating set $S\times \Z_r$ of $Y$ appear in $H_i$ precisely once.


\

	{\bf Step 3.}
The final step consists of modifying  each $H_{ij}$  into $H_{ij}'$, which will be isomorphic to the original tree $T$, in such a way that,
	$$
	H(r)=\oplus_{0\le i,j< r} H'_{ij}.
	$$
Each arc $(x,y)$ in $H$ is splited in $H(r)$ into a (oriented) complete bipartite graph $K_{r,r}$ that we denote by $K_{r,r}^{(x,y)}$.  The $H'_{ij}$ will be constructed by rearranging the arcs in $K_{r,r}^{(x,y)}$ whenever $y$ has indegree two in $H$. This rearrangement of arcs will be performed locally    not   affecting the remaining arcs of $H_{ij}$.
	
\
	
Suppose   that  $y=f_1(u)$, where $y\in V(T_0)$ and  $u\in V(F)$, so that $y$ is incident with a conflicting arc of $H$.
	
Let $x$ be the vertex of $T_0$ adjacent to $y$ in $T_0$ and let $z$ be the vertex of $F$ adjacent to $y$ in $H$ (which creates an undesired cycle; see   Figure \ref{fig:cinc}  for an illustration.)
	
\begin{figure}[h]
\small{
\begin{tikzpicture}[scale=0.4]
\foreach \i in {1,...,13}
{\path (27.5*\i:4cm) coordinate (P\i);
				\path (27.5*\i+4:4cm) coordinate (Q\i);
				\path (27.5*\i-4:4cm) coordinate (R\i);
				\draw[fill]  (Q\i) circle (2pt);
				\draw[fill]  (R\i) circle (2pt);}
\foreach \i in {1,...,13}
{\foreach \j in {\i,...,13}
{\draw[lightgray] (Q\i)--(Q\j) (Q\i)--(R\j) (R\i)--(R\j);}}
			\foreach \i in {1,...,13}
			{\draw[fill]  (Q\i) circle (2pt);
				\draw[fill]  (R\i) circle (2pt);}
			\draw  (Q3)--(R2) (Q3)--(Q2)   (R3)--(Q2) (R3)--(R2);
			\draw   (Q2)--(R12) (Q2)--(Q12)   (R2)--(Q12) (R2)--(R12);
			\draw    (Q12)--(R10) (Q12)--(Q10) (R12)--(Q10) (R12)--(R10);
			\draw   (Q10)--(Q6) (Q10)--(R6)  (R10)--(Q6)(R10)--(R6);
			\draw   (Q12)--(R7) (Q12)--(Q7)   (R12)--(Q7) (R12)--(R7);
			\draw   (Q10)--(R3) (Q10)--(Q3)  (R10)--(Q3) (R10)--(R3);
			\node at (-4,3) {$Y$};
			\node[above right] at (Q3) {$x$};
			\node[below right] at (Q12) {$z$};
			\node[below left] at (Q10) {$y$};
			\end{tikzpicture}
			\hspace{5mm}
			\begin{tikzpicture}[scale=0.35]
			\draw[fill,lightgray] (5.5,10) ellipse (2.2cm and 0.4cm);
			\draw[fill,lightgray] (5.5,0) ellipse (2.2cm and 0.4cm);
			\draw[fill,lightgray] (0,5) ellipse (0.4cm and 2.2cm);
			\path (0,6.5) coordinate (P1);
			\path (0,3.5) coordinate (P2);
			\path (4,10) coordinate (Q1);
			\path (7,10) coordinate (Q2);
			\path (4,0) coordinate (R1);
			\path (7,0) coordinate (R2);
			\draw[fill] (P1) circle (3pt);
			\draw[fill] (P2) circle (3pt);
			\draw[fill] (Q1) circle (3pt);
			\draw[fill] (Q2) circle (3pt);
			\draw[fill] (R1) circle (3pt);
			\draw[fill] (R2) circle (3pt);
			\draw[reverse directed] (P1)--(Q1);
			\draw[directed] (Q2)--(P1);
			\draw[reverse directed]  (P2)--(Q1);
			\draw[reverse directed]  (P2)--(Q2);
			\draw[reverse directed] (P1)--(R1);
			\draw[directed]  (R2)--(P1);
			\draw[reverse directed]  (P2)--(R1);
			\draw[directed]  (R2)--(P2);
			\draw[reverse directed] (P1)--(Q1);
			\draw[reverse directed] (P1)--(R1);
			\node [left] at (P1) {$(y,0)$};
			\node [left] at (P2) {$(y,1)$};
			\node [above] at (Q1) {$(x,0)$};
			\node [above] at (Q2) {$(x,1)$};
			\node [below] at (R1) {$(z,0)$};
			\node [below] at (R2) {$(z,1)$};
			\node at (1.8,9) {$H_1$};
			\node at (5.9,7) {$K_{r,r}^{(x,y)}$};
			\node at (5.9,3) {$K_{r,r}^{(z,y)}$};
			\node at (2.8,3) {$H_1$};
			\node at (14,5) {$
				\begin{array}{c|cc|cc}
				& (x,0)& (x,1) & (z,0) & (z,1) \\ \hline
				(y,0) & {1} & {2} & {1} & {4} \\
				(y,1) & {3} & {4} & 3 & 2
				\end{array}$};
\end{tikzpicture}}
\caption{Conflicting arcs at $y$. }
\label{fig:cinc}
\end{figure}
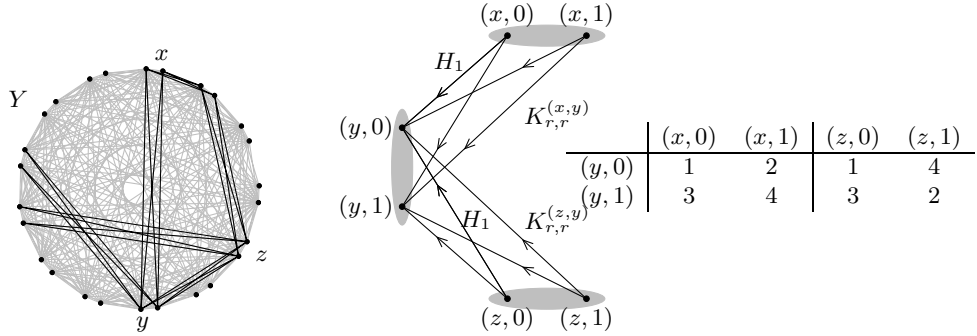
	
	Each edge in $K_{r,r}^{(x,y)}$ belongs to one of $r^2$ trees $T_{ij}$ isomorphic to $T_0$  in the decomposition of $H(r)$ and likewise, each  edge in $K_{r,r}^{(z,y)}$ belongs to one of the  $F_{ij}$. For simplicity we label these copies  with the numbers $(i+1)+rj\in \{1,2,\ldots ,r^2\}$.
	
	We construct the   $r\times r$ matrix $M_x$  by placing at the entry $0\le i,j<r$ the label of  the copy  $T_{ij}$ of $T_0$ which contains the arc $((x,i),(y,j))$. Likewise the $r\times r$ matrix $M_z$ has the number $(i'+1)+rj'$ in the entry $(i,j)$ if   $((z,i), (y,j))$ belongs to $F_{i'j'}$ . Without loss of generality we may assume that
$$(M_x, M_z)=\left(
	\begin{array}{cccc|cccc}
	1 & 2 &\cdots & r &\sigma_1 & \sigma_2 &\cdots & \sigma_r \\
	r+1& r+2& \cdots & 2r& \sigma_{r+1}& \sigma_{r+2}& \cdots & \sigma_{2r} \\
	\vdots & &  & \vdots & \vdots & &  & \vdots\\
	r(r-1)+1&r(r-1)+2&\cdots &r^2&  \sigma_{r(r-1)+1}&\sigma_{r(r-1)+2}&\cdots &\sigma_{r^2}
	\end{array}
	\right),$$
for some permutation $\sigma=(\sigma_1,\ldots ,\sigma_{r^2})$ of $\{1,\ldots ,r^2\}$.
	
	If all the rows of $(M_x, M_z)$ have pairwise distinct entries, then no vertex in $y\times \Z_r$ belongs to a cycle of  $H_{ij}$.  Thus our goal is to redistribute the edges of $K_{r,r}^{(x,y)}$ and/or $K_{r,r}^{(z,y)}$ among the $H_{ij}$ in such a way that the  resulting  matrix $(M_x', M_z')$ has no rows with repeated entries.
	
	Every permutation of the entries in one column of $M_x$ and in one column of $M_z$ does not affect the fact that we have an edge decomposition of $H(r)$.
	
	By Lemma \ref{lem:mamb}, there is a matrix $M'=(M'_x, M'_z)$ obtained from $M$ by permuting the entries within columns which have no repeated entries in the same row.  We assign the arcs to the numbered trees and forests according to the new matrix $M'=(M'_x, M'_z)$. By doing so, there are no conflicting arcs incident to vertices in $y\times \Z_r$.
	
	This local rearrangement may have affected the connectivity of the copies of $T_{ij}$ on the arcs coming out from vertices in $y\times \Z_r$. It may happen that some $H_{i,j}$ is no longer incident to the same vertex in $y\times \Z_r$ as the vertex incident from the next arc coming out in the same graph (see Figure \ref{fig:after} for an illustration.)
	
	\begin{figure}[h]
		\begin{center}
			\begin{tikzpicture}[scale=0.35]
			\draw[fill,lightgray] (5.5,10) ellipse (2.2cm and 0.4cm);
			\draw[fill,lightgray] (5.5,0) ellipse (2.2cm and 0.4cm);
			\draw[fill,lightgray] (0,5) ellipse (0.4cm and 2.2cm);
			\draw[fill,lightgray] (-6,5) ellipse (0.4cm and 2.2cm);

			\path (0,6.5) coordinate (P1);
			\path (0,3.5) coordinate (P2);
			\path (4,10) coordinate (Q1);
			\path (7,10) coordinate (Q2);
			\path (4,0) coordinate (R1);
			\path (7,0) coordinate (R2);
			\path (-6,6.5) coordinate (S1);
			\path (-6,3.5) coordinate (S2);
			\draw[fill] (P1) circle (3pt);
			\draw[fill] (P2) circle (3pt);
			\draw[fill] (Q1) circle (3pt);
			\draw[fill] (Q2) circle (3pt);
			\draw[fill] (R1) circle (3pt);
			\draw[fill] (R2) circle (3pt);
			\draw[fill] (S1) circle (3pt);
			\draw[fill] (S2) circle (3pt);
			\draw[reverse directed] (P1)--(Q1);
			\draw[directed] (Q2)--(P1);
			\draw[reverse directed]  (P2)--(Q1);
			\draw[reverse directed]  (P2)--(Q2);
			\draw[reverse directed] (P1)--(R1);
			\draw[reverse directed]  (P1)--(R2);
			\draw[reverse directed]  (P2)--(R1);
			\draw[directed]  (R2)--(P2);
			\draw[reverse directed] (P1)--(Q1);
			\draw[reverse directed] (P1)--(R1);
			\draw[reverse directed] (S1)--(P1);
			\draw[reverse directed] (S2)--(P1);
			\draw[reverse directed] (S1)--(P2);
			\draw[reverse directed] (S2)--(P2);
			\node at (-.5,7.5) {$(y,0)$};
			\node at (-.4,2.5) (P2) {$(y,1)$};
			\node [above] at (Q1) {$(x,0)$};
			\node [above] at (Q2) {$(x,1)$};
			\node [below] at (R1) {$(z,0)$};
			\node [below] at (R2) {$(z,1)$};
			\node at (1.8,9) {$ H_3$};
			\node at (5.8,7) {$K_{r,r}^{(x,y)}$};
			\node at (2.6,6.7) { $ H_1$};
			\node at (5.8,3) {$K_{r,r}^{(z,y)}$};
			\node[above] at (-3,6.5) {$H_1$};
			\node[below] at (-3,3.5) {$H_3$};
			\node at (14,5) {$\begin{array}{c|cc|cc}
				& (x,0)& (x,1) & (z,0) & (z,1) \\ \hline
				(y,0) & {\bf 3} & {2} & {1} & {4} \\
				(y,1) & {\bf 1} & {4} & 3 & 2
				\end{array}$};
			\end{tikzpicture}
		\end{center}
\caption{Distribution of arcs after rearrangment: $H_1$ and $H_3$ have been disconnected.}
\end{figure}
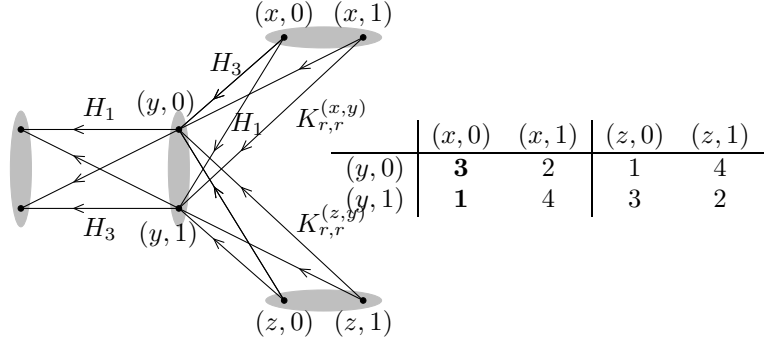\label{fig:after}
	
	In order to repair the connectedness of the $H_{ij}$'s and, at the same time, do not affect these graphs outside the conflicting point in which rearranging of arcs was performed, we reassign the arcs coming out from vertices in $y\times \Z_r$ to copies of $H_{ij}$ so that they originate where the rearrangament placed them and point to the same vertex they did before the arrangement. Such a local reassignment can be always achieved within the $K_{r,r}^{(y,y')}$ corresponding to every arc $(y,y')$ in $H$.
	
	We can make the local arrangements described above by following the original peeling order of $T_0$. We proceed to modify the distribution of the arcs as we encounter vertices incident with conflicting arcs in that order. In this way we travel through directed arcs from the root of each $T_{ij}$, so that rearrangements of arcs do not affect modifications made previously until all conflicting arcs have been processed.
	
	At this point  we  obtain an edge decomposition of $H(r)$ into the $r^2$  oriented graphs  $H'_{i,j}$, each one isomorphic to our given tree $T$. What remains is to use this decomposition of $H(r)$ to produce a decomposition of the whole Cayley graph $Y$.
	
	Let
	$$
	H'_i=\oplus_{0\le j<r} H'_{ij}.
	$$
	Since $H_i$ is a rainbow subgraph of $Y$, the construction of $H'_i$ yields also a rainbow subgraph of $Y$. Therefore, the set of translates
	$$
	\{H'_i+(x,0):  x\in \Z_p\}
	$$
	is edge--disjoint. On the other hand, this set of translates $\{H'_i+(x,0): x\in \Z_p\}$ is vertex disjoint with $\{H'_{i'}+(x,0): x\in\Z_p\}$ with $i'\neq i$. In particular, both graphs are also edge--disjoint. Thus
	$$
	Y=\oplus_{x\in \Z_p} (\oplus_{0\le i<r}H'_i+(x,0))=\oplus_{x\in \Z_p}((\oplus_{0\le i, j<r}H'_{ij})+(x,0))
	$$
	is a decomposition of $Y$ into copies of $T$.  This completes the proof. \qed
\end{proof}


\section{Proofs of the results}\label{sec:proofs}

Lemma \ref{lem:3m} leads directly to a proof of Theorem \ref{thm:main1}.

\

{\bf Proof of Theorem \ref{thm:main1}}. Robinson and Schwenk \cite{robinson} proved that the average number of  leaf s in an (unlabelled) random tree with $m$ edges is asymptotically $cm$ with $c\approx 0.438$.  Drmota and Gittenberger \cite{drmotagitt99} showed that the distribution of the number of  leaf s in a random tree with $m$ edges  is asymptotically normal with variance $c_2m$ for some positive constant $c_2$. Thus, asymptotically almost surely a random tree with $m$ edges has more than $2m/5$  leaf s. It follows from Lemma \ref{lem:3m} that a tree with at least $2m/5$  leaf s decomposes $K_{2m+1}(r)$ for each $r\ge 2$ and $m=(p-1)/2\ge 5$ edges, where $p>10$ is a prime.
\qed

\

Corollary \ref{cor:4m+2}  follows from Theorem \ref{thm:main1} with $r=2$, because $K_{2m+1}(2)$ is isomorphic to $K_{4m+2}\setminus M$, for $M$ any matching of $K_{4m+2}$.
\qed

\

We next  prove Corollary \ref{cor:6m+5}.

\

{\bf Proof of Corollary \ref{cor:6m+5}.}
Let $T$ be  a random tree with $m+1$ edges. Let $T'$ be the tree obtained from $T$ by deleting one  leaf  $yz$, where $z$ is an end vertex of $T$.  As explained in the proof of Theorem \ref{thm:main1}, we know that (a.a.s.) the tree $T'$ has at least $\frac{2m}{5}$ end vertices.

In what follows we use the notation from the proof of Lemma \ref{lem:3m}. Consider the decomposition of  the Cayley graph $Y=Cay(\Z_{2m+1}\times \Z_3, S\times \Z_3)$  into copies of $T'$ (with its arcs oriented from the root to the  leaf s of a peeling ordering of $T'$) which is described in that proof. It remains to complete each copy to $T$ by adding the missing  leaf .

To this end we add two additional vertices $\alpha, \beta$ to $Y$ and make them adjacent from every vertex in  $Y$. Moreover we add to $Y$ an oriented triangle in each  stable set of $Y$.  The resulting graph $Y'$ (omitting orientations and colours) is isomorphic to $K_{6m+5}\setminus e$, where $e=\{\alpha, \beta\}$. We next describe how to complete each copy of $T'$ to $T=T'+yz$. We consider two cases.

Suppose first that $y$ is not incident to a conflicting arc in $H=f(T')$. In this case  each of $(y,0), (y,1), (y,2)$ has indegree three in $H(3)$. We assign the three outgoing arcs added to $Y$ from each  $(y,j)$ to one of its three incoming trees bijectively.  By repeating this procedure to each translate  $H(3)+(z,0)$, $z\in \Z_p$, in $Y$ we  obtain a decomposition of $K_{6m+5}\setminus e$,  into copies of $T$. This completes the proof in this case.

Suppose now that $y$ is  incident to a conflicting arc in $H=f(T')$. Since in this case $y$ has indegree two in $H$,  each vertex $(y,j)$ has indegree six in $H(3)$ and, by the construction of the $H'_{ij}$,  the six incoming arcs belong to  six different trees (this was actually the purpose in the construction of $H'_{ij}$.) There are nine trees in total incident to the three vertices $(y,0), (y,1), (y,2)$ in $H(3)$, label them $T_1',\ldots ,T'_9$.   Following the notation from  the proof of Lemma \ref{lem:3m}, we may assume that the row $j$ of the  matrix
$$
(M_x, M_z)=\left(
      \begin{array}{ccc|ccc}
        1 & 2 &3 &\sigma_1 & \sigma_2 & \sigma_3 \\
        4& 5& 6& \sigma_{4}& \sigma_{5}&  \sigma_{6} \\
       7&8&9&  \sigma_{7}&\sigma_{8}& \sigma_{9}
      \end{array}
    \right),
$$
denotes the labels of the trees incident with $(y,j)$, where $\sigma=(\sigma_1,\ldots ,\sigma_9)$ is a permutation of $\{1,\ldots ,9\}$ and  each row has no repeated entries (as it was shown in the proof of Lemma 5.) Figure \ref{fig:assignment} illustrates the situation.

\begin{figure}[h]
\begin{center}
\begin{tikzpicture}[scale=0.4]
\path (12,9) coordinate (P);
\path (12,3) coordinate (Q);
\path (16,6) coordinate (R);
\path (6,8) coordinate (a);
\path (6,4) coordinate (b);
\draw[fill, lightgray] (13.5,6) ellipse (3cm and 4cm);
\draw[fill] (P) circle (2pt);
\draw[fill] (Q) circle (2pt);
\draw[fill] (R) circle (2pt);
\draw[fill] (a) circle (2pt);
\draw[fill] (b) circle (2pt);

\node[left]  at (12.3,6) {$T_1$};
\node at (9,9) {$T_2$};
\node at (8.8,6.8) {$T_3$};
\foreach \i in {1,2,3,5,6,7}
{
\path (\i +8,12) coordinate (P\i);
\path (\i+8,0) coordinate (Q\i);
}
\foreach \i in {1,2,3,5,6,7}
{
\draw[fill] (P\i) circle (2pt);
\draw[fill] (Q\i) circle (2pt);
}
\node[above]  at (P1) {$T_1$};
\node[above]  at (P2) {$T_2$};
\node[above]  at (P3) {$T_3$};
\node[above]  at (P5) {$T_4$};
\node[above]  at (P6) {$T_5$};
\node[above]  at (P7) {$T_6$};
\node[below]  at (Q1) {$T_4$};
\node[below]  at (Q2) {$T_5$};
\node[below]  at (Q3) {$T_6$};
\node[below]  at (Q5) {$T_7$};
\node[below]  at (Q6) {$T_8$};
\node[below]  at (Q7) {$T_9$};
\foreach \i in {1,2,3}
{
\path (19, \i +6) coordinate (R\i);
\draw[fill] (R\i) circle (2pt);}
\foreach \i in {5,6,7}
{
\path (19, \i -2) coordinate (R\i);
\draw[fill] (R\i) circle (2pt);}
\node[right]  at (R1) {$T_7$};
\node[right]  at (R2) {$T_8$};
\node[right]  at (R3) {$T_9$};
\node[right]  at (R5) {$T_1$};
\node[right]  at (R6) {$T_2$};
\node[right]  at (R7) {$T_3$};
\foreach \i in {1,2,3}
{
\draw[directed] (P\i)--(P);
\draw [directed](Q\i)--(Q);
\draw [directed](R\i)--(R);
}
\foreach \i in {5,6,7}
{
\draw[directed,dashed] (P\i)--(P);
\draw[directed,dashed] (Q\i)--(Q);
\draw[directed,dashed] (R\i)--(R);
}
\node[below]  at (Q1) {$T_4$};
\node[below]  at (Q2) {$T_5$};
\node[below]  at (Q3) {$T_6$};
\node[below]  at (Q5) {$T_7$};
\node[below]  at (Q6) {$T_8$};
\node[below]  at (Q7) {$T_9$};
\node[left]  at (a) {$\alpha$};
\node[left]  at (b) {$\beta$};
\draw[reverse directed] (Q)--(P);
\draw[directed] (Q)--(R);
\draw[directed] (R)--(P);
\draw[directed]  (P)--(a);
\draw[directed] (P)--(b);
\draw[directed] (Q)--(a);
\draw[directed] (Q)--(b);
\draw[directed] (R)--(a);
\draw[directed] (R)--(b);
\node[right] at (P) {$(y,0)$};
\node[right] at (Q) {$(y,1)$};
\node at (16.2,7) {$(y,2)$};
\node[below]  at (Q1) {$T_1$};
\end{tikzpicture}
\end{center}
\caption{Dotted lines indicate the trees labeled in the right matrix $M_z$. A good assignment on $(y,0)$ is shown in the picture. Only one of the two orientations of the triangle admits a good assignment in this example.}
\label{fig:assignment}
\end{figure}
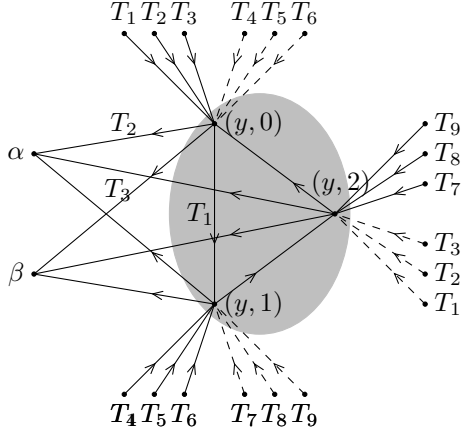

Consider the orientation $(y,0),(y,1),(y,2)$ of the  triangle induced by these three vertices in  $Y'$.
We assign the three outgoing arcs added to $Y$ from each  $(y,j)$ to one of its three incoming trees in row $j$ of $M_x$  bijectively.  There is one such bijection which avoids creating a cycle with the   incoming trees at $(y,j+1)$ unless  the same trees are listed in row $j+1 \pmod{3}$ of $M_z$, in which case we can instead create a cycle.   In this situation, up to a permutation, the matrix $(M_x,M_z)$ would have the form
$$
(M_x, M_z)=\left(
      \begin{array}{ccc|ccc}
        1 & 2 &3 &\sigma_1 & \sigma_2 & \sigma_3\\
        4& 5& 6& 1& 2& 3\\
       7&8&9&  \sigma_{7}&\sigma_{8}& \sigma_{9}
      \end{array}
    \right),
$$
Then,  the reverse orientation $(y,0),(y,2),(y,1)$ admits a good assignment avoiding undesired cycles.  Indeed, if the entries of some row $j$ are the same as the entries of row $j-1$ (modulo $3$) of the above matrix, then the remaining row would   contain the same entries in $M_x$ as the ones in $M_z$, 
$$(M_x, M_z)=\left(
      \begin{array}{ccc|ccc}
        1 & 2 &3 & 4 & 5 & 6\\
        4& 5& 6& 1& 2& 3\\
       7&8&9&  7 & 8 & 9
      \end{array}
    \right),$$
contradicting that the matrix has no repeated entries in each row. This completes the proof.
 \qed

\ 

The argument used in the above proof can be extended to prove Corollary \ref{cor:r>3}.

\

{\bf Proof of Corollary \ref{cor:r>3}:} We imitate the proof of Corollary \ref{cor:6m+5}. Choose an end vertex $y$ of $T$ and   delete the leaf $xy$  so that the resulting tree $T'$ has $m$ edges and at least $2m/5$ end vertices. By Lemma  \ref{lem:3m} we obtain a decomposition of $Y=Cay(\Z_{2m+1}\times \Z_r, S\times \Z_r)$ by copies of an orientation of $T'$.

Consider the oriented graph $Y'$ obtained from $Y$ by adding $(r+1)/2$ new vertices $\alpha_1,\ldots ,\alpha_{(r+1)/2}$   and all arcs from $Y$ to these vertices. Moreover we insert a regular tournament $T_r$ in each stable set of $Y$.  By removing the orientations, $Y'$ is isomorphic to $K_{r(2m+1)+\frac{r+1}{2}}\setminus K_{(r+1)/2}$ (the vertices form a stable set in $Y'$.)

Each vertex $x$ in $H(r)$ is incident to $r$ copies of $T'$ if  $x$ is adjacent to $y$ in $T$.

We next add one leaf to each copy of $T'$  by using the $(r+1)/2$ arcs to $\alpha_1,\ldots ,\alpha_{(r+1)/2}$ and the $(r-1)/2$ arcs in the regular tournament through that vertex. This results in $r$ copies of $T$ in $Y'$ if there were not conflicting arcs of the oriented graph $H$ used to obtain the decomposition of $Y$. If this was not the case, then there is some regular tournament   which admits  a good assignment in the sense described in the proof of Corollary \ref{cor:6m+5}. We omit the details of this last statement.
\qed

\section*{Acknowledgements} This work  is partially supported by the Spanish Ministerio de Economia y Competitividad, under grant MTM2014-60127-P.
 I am grateful to Guillem Perarnau for his careful reading of a preliminary version of this paper.

\end{document}